\documentclass[leqno,11pt]{amsart}
\usepackage[top=25mm, bottom=25mm, left=25mm, right=25mm]{geometry}
\usepackage{amssymb}
\usepackage{dsfont}
\usepackage{mathtools}
\usepackage{enumitem}

\usepackage{color}

\usepackage[hidelinks]{hyperref}
\usepackage{nameref}
\hypersetup{colorlinks = true, urlcolor = blue, linkcolor = blue, citecolor = red}

\numberwithin{equation}{section}
\newtheorem{theorem}{Theorem}
\numberwithin{theorem}{section}
\newtheorem{lemma}[theorem]{Lemma}
\newtheorem{proposition}[theorem]{Proposition}

\newtheorem{conjecture}{Conjecture}
\newtheorem{remark}[theorem]{Remark}

\newcommand{\KK}{\mathbb{K}}
\newcommand{\RR}{\mathbb{R}}
\newcommand{\ZZ}{\mathbb{Z}}
\newcommand{\CC}{\mathbb{C}}
\newcommand{\NN}{\mathbb{N}}

\newcommand{\Var}{{\rm Var}}
\newcommand{\BV}{{\rm BV}}

\newcommand{\ind}[1]{{\mathds{1}_{{#1}}}}

\makeatletter
\@namedef{subjclassname@2020}{%
	\textup{2020} Mathematics Subject Classification}
\makeatother

\begin{document}
	
	\title[Variation of the centered maximal operator]{Variation of the one-dimensional centered maximal \\ operator on simple functions with gaps between pieces}
	
	
	\author[P. Hagelstein]{Paul Hagelstein}
	\address{Paul Hagelstein (\textnormal{Paul\_Hagelstein@baylor.edu}) \newline
		Baylor University, Waco, Texas 76798, USA} 
	
	\author{Dariusz Kosz}
	\address{ Dariusz Kosz (\textnormal{Dariusz.Kosz@pwr.edu.pl}) \newline
	Wroc{\l}aw University of Science and Technology,
	50-370 Wrocław, Poland
	}

	\author[K. Stempak]{Krzysztof Stempak}
	\address{Krzysztof Stempak (\textnormal{Krz.Stempak@gmail.com}) \newline
	55-093 Kie\l{}cz\'ow, Poland}    
	
	\begin{abstract}
	Let $M$ denote the centered Hardy--Littlewood operator on $\RR$. We prove that
	\[
	\Var (Mf)\le \Var (f) - \frac12\big| |f(\infty)|-|f(-\infty)|\big|
	\]
	for piecewise constant functions $f$ with nonzero and zero values alternating. The above inequality strengthens a recent result of Bilz and Weigt \cite{BW} proved for
	indicator functions of bounded variation vanishing at $\pm\infty$.
	We conjecture that the inequality holds for all functions of bounded variation, representing a stronger version of the existing conjecture $\Var (Mf)\le \Var (f)$. We also obtain the discrete
	counterpart of our theorem, moreover proving a transference result on equivalency between both settings that is of independent interest.
	\end{abstract}
	
	\maketitle		
	
	\section{Introduction} \label{S1}
	 
	Regularity properties of maximal operators have been investigated in numerous papers and various frameworks. Kinnunen \cite{Kinnunen1997} initiated the study
	of boundedness of the Hardy--Littlewood maximal operator on the Sobolev spaces $W^{1,p}(\RR^n)$ for $p \in (1,\infty)$. Subsequently, \cite{Kinnunen1997} was complemented
	by Kinnunen and Lindqvist \cite{KiLi1998} to the setting of open subsets of $\RR^n$. See also Tanaka \cite{Tanaka2002} for the case $p=1$ and $n=1$.
	Starting from the Euclidean setting and the Hardy--Littlewood maximal operators, both centered and uncentered, the study then encompassed their variants
	such as fractional maximal operators or maximal operators of convolution type; see \cite{CS2013}, \cite{CMP2017}, \cite{Liu2017}, \cite{M2017}, and references therein.
	Also discrete analogues of these operators were considered, and related questions in both settings have been recently studied; see \cite{M2019}, \cite{GRK2021}, \cite{CGRM2022}, \cite{BGRMW2023}, \cite{GR2023}.
	
	It was an important observation of Aldaz and P{\'e}rez L{\'a}zaro \cite{AP2007} that maximal operators can actually improve the regularity of involved 
	functions rather than simply preserve it. They also pointed out the role to be played by the variation when measuring the regularity. 
	In particular, it was proved in \cite{AP2007} that if $f\in\BV(\RR)$, then for the uncentered Hardy--Littlewood maximal operator $\widetilde{M}$ 
	the function $\widetilde{M}f$ is absolutely continuous on $\RR$ and $\Var(\widetilde{M}f)\le \Var(f)$. For the centered operator $M$, on the other hand, Kurka \cite{Kurka2010} proved
	\[
	\Var(Mf)\le C\Var(f), \qquad f\in\BV(\RR),
	\]
	with some, quite large, $C>0$. It was conjectured that also in this case $C=1$ is enough, see e.g. \cite[Section~1]{Kurka2010}. In the discrete case, 
	i.e., for $f\in\BV(\ZZ)$ and the discrete counterpart of $M$, 
	the same was conjectured, see \cite[Question~B]{BCHP2012}. It is worth recalling the common truth that the  uncentered maximal operator has better 
	regularity behavior than its centered counterpart. Furthermore, the proofs of analogous results are usually much subtler in the centered setting. 
	
	It seems that both conjectures, although reasonable and expected to be true, do not take the full advantage of how the limits $f(\pm\infty)$ or $F(\pm\infty)$ 
	for $f\in\BV(\RR)$ or $F\in\BV(\ZZ)$, respectively, determine the limits $Mf(\pm \infty)$ or $MF(\pm \infty)$; see Section~\ref{S3}. We thus propose the following strengthening, which would complement analogous inequalities for other maximal operators; see Remark~\ref{R2}. 
		
	\begin{conjecture} \label{conj2}
		Suppose $f\in\BV(\RR)$. Then
		\begin{equation} \label{eq:conj2}
		\Var (Mf) \leq \Var (f) - \frac12 \big||f(\infty)|-|f(-\infty)|\big|.
		\end{equation}
	\end{conjecture}
	
	Our first main result, Theorem~\ref{T1}, supports this conjecture. We mention that there is no need to formulate separately the discrete 
	counterpart of Conjecture~\ref{conj2}, since by our second main result, Theorem~\ref{T2}, inequality \eqref{eq:conj2} and its discrete analogue are equivalent. Clearly, verification of either version of
	the conjecture, continuous or discrete, reduces to the consideration of nonnegative functions. We shall prove that a certain further reduction to nonnegative simple functions is available; see Proposition~\ref{T3}. Finally, we show that the constants $1$ 
	and $\frac12$ in \eqref{eq:conj2} are optimal; see Section~\ref{S4}.
	
	Let $\NN \coloneqq \{1,2,\dots\}$. Let $\KK$ be either $\RR$ or $\ZZ$. For $g \colon \KK \to \CC$ and any nonempty $E \subseteq \KK$ define
	\[
	\Var_E (g) \coloneqq
	\sup_{\text{monotone } \phi \colon \ZZ \to E} \sum_{m \in \ZZ} \big|g\big(\phi(m+1)\big) - g\big(\phi(m)\big)\big|, 
	\]
	which we call the \emph{variation of} $g$ \emph{over} $E$.
	We write $g \in \BV(\KK)$ when the \emph{total variation} $\Var (g) \coloneqq \Var_\KK(g)$ is finite. For each such $g$ the limits 
	$g(\pm\infty) \coloneqq \lim_{x \to \pm \infty} g(x)$ exist and, when $\KK = \RR$, for each $x \in \RR$ the one-sided limits $g(x^\pm) \coloneqq \lim_{y \to x^\pm} g(y)$ 
	exist as well. 
	
	Recall the definition of the centered Hardy--Littlewood maximal operator $M$. We set
	\[
	Mg(x) \coloneqq \sup_{r>0} \frac{1}{2r} \int_{x-r}^{x+r} |g|
	\quad \text{and} \quad
	MG(n) \coloneqq \sup_{m\ge0} \frac{1}{2m+1} \sum_{y  = n - m}^{n + m} |G(y)|
	\]
	for any locally integrable $g \colon \RR \to \CC$ and $x \in \RR$, and for any $G \colon \ZZ \to \CC$ and $n \in \ZZ$. Of course, if $g \in \BV(\RR)$, then $g$ is bounded, and hence so is $Mg$. The same is true for $MG$ when $G \in \BV(\ZZ)$.
	
	Our first main result says that \eqref{eq:conj2} holds for a certain special class of functions. 
	
	\begin{theorem} \label{T1}
		Let  $-\infty<A<B < +\infty$, $a,b\in\CC$, $K \in \NN$, and $\{\alpha_k\}_1^K \subset\CC$ such that $\alpha_k \alpha_{k+1} = 0$, $1\le k\le K-1$, be given.
		For a system $-\infty<A=t_0<t_1<\dots <t_K=B < +\infty$
		define
		\[
		f(x) \coloneqq a \ind{(-\infty,A)}(x) + \sum_{k=1}^K \alpha_k \ind{[t_{k-1},t_k)}(x) + b \ind{[B,\infty)}(x)
		\]
		for all $x \in \RR$. Then 
		\begin{equation} \label{eq:T1}
		\Var (Mf) \leq \Var (f) - \frac12 \big||a|-|b|\big|.
		\end{equation} 
		
		Similarly, replacing the system of real numbers  $\{t_k\}_0^K$ by a system of integers $-\infty<A=n_0<n_1<\dots <n_K=B < +\infty$, and assuming $\alpha_k \alpha_{k+1} = 0$, $1\le k\le K-1$, define
		\[
		F(n) \coloneqq a \ind{(-\infty,A)}(n) + \sum_{k=1}^K \alpha_k \ind{[n_{k-1},n_k)}(n) + b \ind{[B,\infty)}(n)
		\]
		for all $n \in \ZZ$. 
		Then
		\begin{equation} \label{eq:T1a}
		\Var(MF) \leq \Var(F) -  \frac12 \big||a|-|b|\big|.
		\end{equation} 
	\end{theorem}
	
	Theorem~\ref{T1} implies \eqref{eq:T1} or \eqref{eq:T1a} for all indicator functions of bounded variation, $f \colon \RR \to \{0,1\}$ or $F \colon \ZZ \to \{0,1\}$, respectively, thus strengthening
	\cite[Theorems~1.1~and~1.3]{BW}.
	On the other hand, Theorem~\ref{T1} does not recover \cite[Theorem~1.2]{BW}, which asserts that $\Var (Mf) \leq \Var (f)$ holds for every $f \colon \RR \to [0,\infty)$ satisfying $f(x) = 0$ or $f(x)=Mf(x)$ for almost every $x \in \RR$.
	
	Notice the \emph{separation condition} $\alpha_k \alpha_{k+1} = 0$ encoded in the systems  
	$\{t_k\}_0^K$ and $\{n_k\}_0^K$.
	In fact, our techniques allow us to skip this condition at the expense of using $|\alpha_{k}| + |\alpha_{k+1}|$, not $|\alpha_{k+1}-\alpha_k|$, on the right-hand side of  \eqref{eq:T1} or \eqref{eq:T1a}. For example, if $\KK = \ZZ$ and $a=b=0$, then \eqref{eq:T1a} would take the form $\Var(MF) \leq 2(|\alpha_1| + \dots + |\alpha_K|)$. Building on this, one can recover \cite[Theorem~1.1]{M2017}, which asserts that $\Var(MF) \leq 2 \| F \|_{\ell^1(\ZZ)}$ holds for all $F \in \ell^1(\ZZ)$. 
	
	Moreover, by combining Theorem~\ref{T1} with \cite[Subsection~5.3]{R2019}, we obtain a variant of \eqref{eq:T1} with $M$ replaced by $M^{\theta}$ for any $\theta > 0$, where $M^\theta$ is the nontangential maximal operator studied in \cite{R2019}.
	
	Lastly, our methods differ from \cite{BW} and align more closely with \cite{M2017}. In particular, the approach in Lemma~\ref{L3} resembles \cite[Section~2]{M2017}, though it gains additional efficiency through Lemma~\ref{L2}.        
	
	Throughout the paper, $I$ with possible subscripts and/or superscripts affixed will denote an open bounded subinterval of $\RR$. Then $c(I), l(I), r(I)$ will be the center and the left/right endpoint of  $I$, respectively.
	As already mentioned, $\Var$ with no subscripts, $\RR$ or $\ZZ$, will denote the total variation; it will be clear from the context to which situation this symbol refers to. Appearance of a subscript will indicate that the variation is related to a subset represented by this subscript.

	\section{Transference} \label{S2}
	
	The main goal of this section is to prove a transference result, Theorem~\ref{T2}, which seems to be of independent interest.  We begin with a simple result that exhibits a substantial improvement of regularity of $M$. Namely, given $g \in \BV(\RR)$, not only does $Mg$ have one-sided limits at every $x\in\RR$ but it is one-sided continuous as well. Of course, the first claim follows, since $Mg \in \BV(\RR)$, cf.~\cite{Kurka2010}, but we decided to provide a short proof of this fact to make the whole argument self-contained.
	
	\begin{lemma} \label{limits}
	For a generic function $g \in \BV(\RR)$ and $x \in \RR$ the limits $Mg(x^\pm)$ exist and we have
	\begin{equation} \label{eq:T2}
	Mg(x) = \min \{ Mg(x^-), Mg(x^+) \}.
	\end{equation}	
	\end{lemma}
	
	\begin{proof}
	We can assume $g \geq 0$. Regarding the existence of the limits, fix $\varepsilon > 0$ and take $\delta > 0$ such that $|g(y) - g(x^-)| < \varepsilon$ for $y \in (x- 2 \delta, x)$ and $|g(y) - g(x^+)|  < \varepsilon$ for $y \in (x, x + 2 \delta)$. Let
	$\tilde g(y)$ be $g(x^-)$, $g(x^+)$, or $g(y)$, if $y$ belongs to $(x-2\delta, x)$, $[x, x+2\delta)$, or $\RR\setminus (x-2\delta, x+2\delta)$, respectively. Then $|Mg(y) - M\tilde g(y)| \leq \varepsilon$ for $y \in \RR$. Note that $M \tilde g = \max\{M_{< \delta} \tilde g, M_{\geq \delta} \tilde g\}$, where $M_{< \delta}$ and $M_{\geq \delta}$ allow only the centered intervals $I$ with $|I| < 2 \delta$ and $|I| \geq 2 \delta$, respectively. If $y \in (x, x+\delta)$, then 
	\[
	M_{<\delta} \tilde g(y) = \max \big\{ g(x^+), (2\delta)^{-1} \big((\delta - y+x)g(x^-) + (\delta + y - x)g(x^+) \big)  \big\},
	\]
	whereas if $x < y < z < x+\delta$, then 
	\[
	|M_{\geq \delta} \tilde g(z) - M_{\geq \delta} \tilde g(y)| \leq (2\delta)^{-1} (z-y)\Var(g),
	\] 
	since for $r \geq \delta$ we have
	\begin{align*}
	\Big| \frac{1}{2r} \int_{z-r}^{z+r} \tilde g - \frac{1}{2r} \int_{y-r}^{y+r} \tilde g \Big|
	=\int_{y+r}^{z+r}\frac{|\tilde g(t)- \tilde g(t-2r)|}{2r} \,dt \leq \frac{(z-y) \Var(g)}{2r}
	\leq \frac{(z-y)\Var(g)}{2 \delta}.
	\end{align*}  
	Thus, $M \tilde g(x^+)$ exists and so $\limsup_{t \to x^+} Mg(t) - \liminf_{t \to x^+} Mg(t) \leq 2 \varepsilon$. Letting $\varepsilon \to 0$, we obtain that $Mg(x^+)$ exists as well. Analogously we verify that $Mg(x^-)$ exists.
	
	Regarding \eqref{eq:T2}, we have $Mg(x) \leq \min \{ Mg(x^-), Mg(x^+) \}$ by lower semicontinuity of $Mg$. To prove the opposite inequality, it suffices to consider the case
	$\min \{ Mg(x^-), Mg(x^+) \} > \frac{g(x^-) + g(x^+)}{2}$ because 
	$
	Mg(x) \geq \frac{g(x^-) + g(x^+)}{2}.
	$ 
	Assume $g(x^-) \leq g(x^+)$ (the case $g(x^-) \geq  g(x^+)$ is symmetric). Now 
	$Mg(x^-) > \max \{ g(x^-), \frac{g(x^-) + g(x^+)}{2} \}$ and so $Mg(y) = M_{\geq \delta} g(y)$ for some $\delta > 0$ and all $y \in (x-\delta, x)$. Since, as before, 
	$|M_{\geq \delta} g(x) - M_{\geq \delta} g(y)| \leq (2 \delta)^{-1} (x-y) \Var(g)$ for all such $y$, 
	the proof is complete.
	\end{proof}
	
	Now, for $\KK=\RR$ or $\KK=\ZZ$, and  $a,b \in \CC$, define
	\[
	\BV_a^b(\KK) \coloneqq \{g \in \BV(\KK) : g(-\infty) = a, \, g(+\infty) = b \}.
	\]
	By $C(\KK)$ we denote the smallest value $C \in (0,\infty]$ such that
	\[
	\Var (Mg) \leq C \Var (g), \qquad g \in \BV(\KK),
	\]
	and by $c_a^b(\KK)$ we denote the largest value $c \in [0,\infty)$ such that
	\[
	\Var (Mg) \leq C(\KK) \Var (g) - c, \qquad g \in \BV_a^b(\KK),
	\]
	with the convention $c_a^b(\KK) = 0$ if $C(\KK) = \infty$. 
	
	The transference result that follows complements \cite[Proposition~1.4]{BW}. We include the potential case $C(\KK) = \infty$ to emphasize that the proof works independently of whether $C(\KK)$ is finite or not. Nonetheless, in \cite{Kurka2010} it was shown that $C(\RR) < \infty$, thus Theorem~\ref{T2}, even in the weaker form with the additional assumption $C(\KK) < \infty$, immediately gives $C(\ZZ) < \infty$.  
	
	\begin{theorem} \label{T2}
		We have $C(\ZZ) = C(\RR)$ and $c_a^b(\ZZ) = c_a^b(\RR)$ for all $a, b \in \CC$.
	\end{theorem}
	
	\begin{proof}
	We first show $C(\ZZ) \leq C(\RR)$ and $c_a^b(\ZZ) \geq c_a^b(\RR)$ by using a natural extension procedure, cf.~\cite[Lemma~4.1]{BW}.
	If $C(\RR)=\infty$, then $c_a^b(\RR)=0$ and there is nothing to show. Assume therefore that $C(\RR)<\infty$.
	Given $G \in \BV_a^b(\ZZ)$, consider
	\[
	g \coloneqq \sum_{n \in \ZZ} G(n) \ind{[n-1/2,n+1/2)}.
	\]
	Then $\Var(g) = \Var(G)$ and so $g \in \BV_a^b(\RR)$. Note that $Mg(n) = MG(n)$, $n \in \ZZ$, since, for $m \in \NN$, one of the averages of $|g|$ over the centered intervals $I$
	of lengths $|I| = 2m \pm 1$ dominates those with $|I| \in [2m-1,2m+1]$. Thus, $\Var(MG)=\Var_\ZZ (Mg) \le \Var(Mg)$, and the claimed inequalities hold.
	
	In the opposite direction, we claim that $C(\ZZ)\geq C(\RR)$ and $c_a^b(\ZZ)\leq c_a^b(\RR)$, and use a convenient variant of a sampling procedure.  
	Obviously, we can assume $C(\ZZ)<\infty$.
	Given $g \in \BV_a^b(\RR)$ and $\varepsilon > 0$, in view of \eqref{eq:T2}, there is $N_*=N_*(\varepsilon) \in \NN$ such that for $E_{N_*} \coloneqq 2^{-N_*}\ZZ \cap [-2^{N_*},2^{N_*}]$ we have
	\begin{equation} \label{eq:T3}
	\Var_{E_{N_*}} (Mg) \geq \Var(Mg) - \varepsilon
	\quad \text{or} \quad \Var_{E_{N_*}} (Mg) \geq \varepsilon^{-1},
	\end{equation}
	depending on whether $\Var(Mg)$ is finite or not, respectively. For $N \in \NN$, let $M_{N}$ be the variant of $M$ allowing only the centered intervals $I$ with $|I| \in (2\NN-1)/2^N$. We have $\lim_{N \to \infty }M_N g = M g$ pointwise, since for $x \in \RR$ and $\delta > 0$ there are $0 < \rho_1 < \rho_2 < \infty$ such that $\frac{1}{2\rho} \int_{x-\rho}^{x+\rho} |g| \geq M g - \delta$ for all $\rho \in [\rho_1, \rho_2]$. Thus, there exists $N_0 \in \NN$ such that
	$
	\Var_{E_{N_*}} (M_{N} g) \geq \Var_{E_{N_*}} (M g) - \varepsilon
	$
	for all $N \geq N_0$.
	Consider $\tilde G_N(n) \coloneqq 2^N \int_{(n-1/2)/2^N}^{(n+1/2)/2^N} |g|$ and $G_N(n) \coloneqq 2^N \int_{(n-1/2)/2^N}^{(n+1/2)/2^N} g$. Then
	\begin{equation} \label{eq:T4}
	\max\{\Var(\tilde G_N), \Var(G_N)\} \leq \int_{-1/2}^{1/2} \Var_{(\ZZ + t)/2^N } (g) \, {\rm d}t \leq \Var(g),
	\end{equation}
	so that $\tilde G_N \in \BV_{|a|}^{|b|}(\ZZ)$ and $G_N \in \BV_a^b(\ZZ)$. Also, $MG_N(n) \leq M\tilde G_N(n) = M_N g(n/2^N)$ for $n \in \NN$. 
	
	Now, we show that $MG_N(n) \geq M\tilde G_N(n) - \varepsilon / 2^{2N_*+4}$ for all $n \in E_{N_*}^N \coloneqq \{2^N Q : Q \in E_{N_*}\} \subset \NN$, provided that $N \geq \max\{N_*, N_0\}$ is large enough. 
	Given $Q \in E_{N_*}$, we have 
	\[
	|I_Q|^{-1} \int_{I_Q} |g| \geq M\tilde G_N(2^N Q) - \varepsilon / 2^{2N_*+5}
	\] 
	for some $I_Q$ centered at $Q$. For each $N$ large enough, let $I_Q^N$ be the longest interval centered at $Q$, contained in $I_Q$, with $|I_Q^N| \in (2\NN - 1)/2^N$. Denoting $g_N \coloneqq \sum_{n \in \ZZ} G_N(n) \ind{[(n-1/2)/2^N, (n+1/2)/2^N)}$, we see that $\lim_{N \to \infty} |I_Q^N|^{-1} \int_{I_Q^N} |g_N| = |I_Q|^{-1} \int_{I_Q} |g|$ by the dominated convergence theorem on $I_Q$, because $g$ is continuous almost everywhere and bounded.  
	Since $MG_N(2^N Q) \geq |I_Q^N|^{-1} \int_{I_Q^N} |g_N|$ and $E_{N_*}$ is finite, $MG_N(n) \geq M\tilde G_N(n) - \varepsilon / 2^{2N_*+4}$ for some $N \geq \max\{N_*, N_0\}$, as desired. Thus,
	\begin{equation} \label{eq:T5}
	\Var_{E_{N_*}} (M g) - \varepsilon
	\leq 
	\Var_{E_{N_*}}(M_{N}g) = \Var_{E_{N_*}^N}(M \tilde G_N) \leq
	\Var_{E_{N_*}^N}(M G_N) + \varepsilon
	\leq \Var(M G_N) + \varepsilon.
	\end{equation}
	Assuming $\Var(Mg)=\infty$ gives
	\[
	\varepsilon^{-1} \le \Var_{E_{N_*}} (M g) \le \Var(MG_N) + 2\varepsilon \le C(\ZZ)\Var(g)-c_a^b(\ZZ) + 2\varepsilon,
	\]
	which contradicts our assumption $C(\ZZ)<\infty$. Consequently, $\Var(Mg)<\infty$ (this includes the case $C(\RR)<\infty$). Applying successively
	the first part of \eqref{eq:T3}, \eqref{eq:T5}, the definitions of $C(\ZZ)$ and $c_a^b(\ZZ)$ in the context of $G_N \in \BV_a^b(\ZZ)$, and \eqref{eq:T4} for $G_N$, and then letting $\varepsilon \to 0$, we verify the claim.
	\end{proof}

	Later on, we shall refer to the two methods from the proof of Theorem~\ref{T2} as to the \emph{extension procedure} and the \emph{sampling procedure}, respectively.
	
	Below we show that proving Conjecture~\ref{conj2} reduces to consideration of nonnegative simple functions.
	We call $G \in \BV(\ZZ)$ \emph{simple}, abusing
	slightly the common terminology, if there exist $a,b\in\CC$ and $N \in \NN$ such that $G(-n) = a$ and $G(n) = b$ when $n > N$. Similarly, $g \in \BV(\RR)$ is \emph{simple} if $G(-x) = a$ and $G(x) = b$ when $x > N$, and $g$ is constant on the intervals $[n-\frac{1}{2},n+\frac{1}{2})$, $n \in \ZZ$.
	 
	\begin{proposition} \label{T3}
		Conjecture~\ref{conj2} is true if \eqref{eq:conj2} is satisfied by nonnegative simple functions.
	\end{proposition}
	
	\begin{proof}
		It is convenient to prove a version of the above for the discrete case,  $\KK = \ZZ$, and then apply Theorem~\ref{T2} and the extension procedure to obtain the result for  $\KK = \RR$.  
		
		Given nonnegative $G \in \BV(\ZZ)$ and $\varepsilon > 0$, we have $\Var_{[-N_*,N_*] \cap \ZZ} (M G) \geq \Var (MG) - \varepsilon$ for some $N_* = N_*(\varepsilon) \in \NN$.
		Note that $G \in \BV_a^b(\ZZ)$ for some $a,b\in [ 0,\infty)$. We choose $N = N(\varepsilon, N_*) \in \NN$ such that $|G(-n) - a| < \varepsilon / (5N_*)$ and $|G(n) - b| < \varepsilon / (5N_*)$ 
		if $n > N$.
		Let $G_N(n)$ be equal to $a$, $G(n)$, or $b$, for $n < N$, $n \in [-N,N]$, or $n > N$, respectively. Then
		\[
		\Var_{[-N_*,N_*] \cap \ZZ} (M G_N) \geq \Var_{[-N_*,N_*] \cap \ZZ} (M G) - \varepsilon \geq \Var (MG) - 2\varepsilon,
		\]
		because $|MG_N(n) - MG(n)| \leq \varepsilon / (5N_*)$. Since $\Var (G_N) \leq \Var (G)$, we are done.
	\end{proof}
 
	\section{Proof of Theorem~\ref{T1}} \label{S3}
	
	For a generic function $g \in \BV(\KK)$ one has 
	\[
	Mg \geq \frac{|g(-\infty)|+|g(\infty)|}{2}
	\quad \text{and} \quad 
	Mg(\pm \infty) =  \max \Big \{|g(\pm \infty)|, \frac{|g(-\infty)|+|g(\infty)|}{2} \Big \}.
	\]
	We say that $Mg$ has a \textit{local maximum} at $x\in \KK$ if there exist $y',y''\in\KK$ such that $y' < x < y''$ and 
	\begin{align} \label{max-def}
	\max\{Mg(y'), Mg(y'')\} < \sup\{ Mg(y) : y \in  \KK \cap [y',y'']  \}  = Mg(x);
	\end{align}
	a \textit{local minimum} of $Mg$ at $x$ is defined analogously. 
	
	We are ready to prove Theorem~\ref{T1}. By the extension procedure it suffices to deal with 
	$\KK = \RR$. 
	Indeed, given $F \in \BV_a^b(\ZZ)$ as in Theorem~\ref{T1} with the corresponding system $\{n_k\}_0^K$, define $f \in \BV_a^b(\RR)$  with the
	corresponding system $\{t_k\}_0^K$, $t_k=n_k - \frac{1}{2}$, by $f \coloneqq \sum_{n \in \ZZ} F(n) \ind{[n-1/2,n+1/2)}$. By using the separation condition $\alpha_k \alpha_{k+1} = 0$ for $F$, we verify that $f$ has the form as in Theorem~\ref{T1}. Moreover, $\Var (f) = \Var(F)$ and 
	$\Var (Mf) \geq \Var(MF)$. Then \eqref{eq:T1a} for $F$ follows from \eqref{eq:T1} for $f$.
	
	Until the end of this section we fix $f$ as in Theorem~\ref{T1} and denote $I_k \coloneqq (t_{k-1},t_k)$, $k=1,\ldots,K$, the interiors of the corresponding intervals.
	Without any loss of generality we may consider only nonnegative $f$. Hence, from now on, we assume $a,b\ge0$ and $\alpha_k\ge0$, $k=1,\ldots,K$. Clearly, $Mf$ may not have local maxima at all. However, if $Mf$ has a local maximum at $x$, then the following two cases may occur:
	\begin{enumerate}
		\item $Mf$ is constant on some interval $J \ni x$, i.e., $Mf$ has a local maximum at each point of $J$;
		\item there is no $J$ with this property, i.e., $x$ is \textit{isolated}.  
	\end{enumerate} 
	In the first case, considering $J$ as the longest interval with the declared property, we shall choose exactly one element $x_* \in J$ as a \emph{representative}. In addition every isolated point $x$ from the second case, if any, will also be called a representative. It follows from  Lemma~\ref{L2} that the number of representatives 
	%of points, where $Mf$ has a local maximum, 
	is finite and does not exceed $K^2$. 
	 
	\begin{lemma} \label{L2}
		Let $Mf$ have a local maximum at $x$. Then there exists $x_*$  (either equal to $x$ or such that $Mf$ is constant on the closed interval connecting $x$ and $x_*$) with the property that $Mf(x_*)$ is attained for some interval $I_{x_*}$ which is centered at $x_*$ and such that $l(I_{x_*}) = l(I_{k'})$ and $r(I_{x_*}) = r(I_{k''})$ hold for some $1\le k'\le k''\le K$ satisfying $Mf(x_*) \leq \min \{\alpha_{k'}, \alpha_{k''}\}$.
	\end{lemma} 
	\begin{proof}
		Take any sequence $\{\rho_n\}\subset (0,\infty)$ such that $Mf(x) = \lim_{n \to \infty} \frac{1}{2\rho_n} \int_{x-\rho_n}^{x+\rho_n} f$. Clearly, if $\{\rho_n\}$ has 
		a subsequence diverging to $\infty$, then $Mf(y) \geq Mf(x)$ for all $y \in \RR$. Thus,  $\{\rho_n\}$ is bounded. 
		
		Assume $\{\rho_n\}$ has a subsequence converging to $0$. Then $Mf(x) = \frac{f(x^-) + f(x^+)}{2}$ and so $f(x^-) = f(x^+) = Mf(x)$ by $Mf(x) \geq Mf(x^\pm) \geq f(x^\pm)$, where the first inequality holds because $Mf$ has a~local maximum at $x$. If $x \leq A$, then $Mf(x) = f(x^-) = a = \frac{f(y^-) + f(y^+)}{2} \leq Mf(y)$ for all $y < x$, contradicting \eqref{max-def}. Similarly, $x \geq B$ is impossible. Thus, $x \in I_k$ and $Mf(x) = \alpha_k$ for some $k$. We can replace $x$ by $c(I_k)$, since $Mf(y) \geq Mf(x)$ for all $y \in I_k$ so that $Mf$ is constant 
		on $I_k$ because $Mf$ has a local maximum at $x$. Taking $x_*=c(I_k)$ and $I_{x_*} = I_k$, we obtain the claimed property.
		
		If $\{\rho_n\}$ is separated from $0$, then it contains a subsequence converging to a positive number. Consequently, $Mf(x) = \frac{1}{|I_x|} \int_{I_x} f$ for some $I_x$ centered at $x$. Note that $f\big(l(I_x)^+\big) \geq M f(x)$, since otherwise we could show $M f(y) > Mf(x)$ for some $\delta > 0$ and all $y \in (x, x + \delta)$ by taking $I_y$ with $r(I_y) = r(I_x)$. In particular, $l(I_x) \geq A$ because otherwise we would have $Mf(y) \geq Mf(x)$ for all $y \leq x$, again by taking $I_y$ with $r(I_y) = r(I_x)$. 
		Thus, $l(I_x) \in \big[l(I_{k'}), r(I_{k'})\big)$ for some $k'$ such that $\alpha_{k'} \geq Mf(x)$. Similarly, $f\big(r(I_x)^-\big) \geq M f(x)$ and $r(I_x) \in \big(l(I_{k''}), r(I_{k''})\big]$ for some $k'' \geq k'$ such that $\alpha_{k''} \geq Mf(x)$. Let $x_* = \frac{l(I_{k'}) + r(I_{k''})}{2}$. If $x=x_*$, then $Mf(x)$ is attained for $\big(l(I_{k'}), r(I_{k''})\big)$. If not, then we can show $Mf(y) \geq Mf(x)$ for all $y$ in the closed interval connecting $x$ and $x_*$ by taking $x_\theta = (1-\theta) x + \theta x_*$ and $I_x \subseteq I_{x_\theta} \subseteq \big(l(I_{k'}), r(I_{k''})\big)$ with $l(I_{x_\theta})= (1-\theta) l(I_x) + \theta l(I_{k'})$ and $r(I_{x_\theta} ) = (1-\theta) r(I_x) + \theta r(I_{k''})$ for all $\theta \in [0,1]$. Since $Mf$ has a local maximum at $x$, it is constant on this closed interval and so $Mf(x_*)$ is attained for $\big(l(I_{k'}), r(I_{k''})\big)$, as desired.
		
		From the reasoning above, it follows that $Mf(x_*) \leq \min \{\alpha_{k'}, \alpha_{k''}\}$ in each case.  
	\end{proof} 
	
	From now on, we take each representative $x$ to be of the form $\frac{l(I_{k'}) + r(I_{k''})}{2}$ for  
	$1\le k' \leq k''\le K$ such that $Mf(x) = \frac{1}{|I_x|} \int_{I_x} f$, where $I_x = \big(l(I_{k'}), r(I_{k''})\big)$. 
	Furthermore, if $Mf$ has local maxima and there are $N\ge1$ representatives, say $x_1<x_2<\dots< x_N$, then for  $N\ge2$ and each $2\le n\le N$, 
	we choose $y_n\in (x_{n-1},x_{n})$, where $Mf$ has a local minimum. Then
	\begin{align} \label{var-split}
	\Var (Mf) = \Var_{(-\infty, x_1]}(Mf) + \Big( \sum_{n=2}^N Mf(x_{n-1}) - 2Mf(y_n) + Mf(x_{n}) \Big) + \Var_{[x_N,\infty)}(Mf).
	\end{align}
	%We also choose arbitrary points $y_1 \in (-\infty, x_1)$ and $y_{N+1} \in (x_N, \infty)$.
	%with the middle term disappearing when $N=1$.
	
	Let $T_I(x) \coloneqq \frac{|I|}{|J_I(x)|}$ for all $x\in\mathbb{R}$, where $J_I(x)$ is the smallest interval $J \supseteq I$ with $c(J)=x$. 
	The next result says that the variation of $Mf$ is controlled locally by the variations of $T_{I_k}$, $1 \leq k \leq K$.
	 
	\begin{lemma} \label{L3}
		Suppose $Mf$ has $N\ge1$ representatives $x_1 < \dots < x_N$.  
		Then
		\begin{align*}
		Mf(x_{n-1}) - Mf(y_{n}) \leq \sum_{k=1}^K \alpha_k \Var_{[x_{n-1},y_{n}]} (T_{I_k}) 
		\quad \text{and} \quad 
		Mf(x_n) - Mf(y_n) \leq \sum_{k=1}^K \alpha_k \Var_{[y_n, x_n]} (T_{I_k}) 
		\end{align*}
		for all $2\le n \le N$ and $x_1<y_2<x_2<\dots<x_{N-1}<y_N<x_N$. Similarly,
		\begin{align*}
		Mf(x_{N}) \leq \sum_{k=1}^K \alpha_k \Var_{[x_{N},\infty)} (T_{I_k}) 
		\quad \text{and} \quad 
		Mf(x_1) \leq \sum_{k=1}^K \alpha_k \Var_{(-\infty, x_1]} (T_{I_k}). 
		\end{align*}
		If $\alpha_1 \leq a$, then the first summand in each of the four sums above may be omitted.   
	\end{lemma}  
	
	\begin{proof}
		We shall estimate $Mf(x_n) - Mf(y_n)$ and for $Mf(x_{n-1}) - Mf(y_{n})$ the proof is analogous.
		Given $x_n$, we take $I_{x_n}$ as in the conclusions of Lemma~\ref{L2} with accompanying $1\le k'\le k''\le K$.
		Then, for $I_{y_n}$  centered at $y_n$ and such that $r(I_{y_n}) = r(I_{x_n})$, we have $I_{x_n} \subseteq I_{y_n}$ and so
		\[
		Mf(x_n) - Mf(y_n) \leq 
		\frac{1}{|I_{x_n}|}\int_{I_{x_n}} f - \frac{1}{|I_{y_n}|}\int_{I_{y_n}} f 
		\leq \bigg( \frac{1}{|I_{x_n}|} - \frac{1}{|I_{y_n}|} \bigg) \int_{I_{x_n}} f
		= \sum_{k=k'}^{k''} 
		\bigg( \frac{\alpha_k |I_k|}{|I_{x_n}|} -  \frac{\alpha_k |I_k|}{|I_{y_n}|} \bigg).
		\] 
		Note that if $\alpha_1 \leq a$ and $k'=1$, then $\alpha_1 \geq Mf(x_n)$ by Lemma~\ref{L2}, contradicting \eqref{max-def}. Indeed, we can show $Mf(y) \geq Mf(x_n)$ for all 
		$y < x_n$ by taking $I_y$ with $r(I_y) = r(I_{x_n})$. This explains the final sentence of the statement. Next, we notice that the numerical inequality $ \frac{1}{\rho_2} - \frac{1}{\rho_2 + \rho_0}\le \frac{1}{\rho_1} - \frac{1}{\rho_1 + \rho_0}$ holds for $\rho_0, \rho_1, \rho_2 \in (0,\infty)$ when $\rho_1 \leq \rho_2$. Thus, if either $x_n \leq c(I_k)$ or $c(I_k) \leq y_n$, then 
		\[
		\frac{\alpha_k |I_k|}{|I_{x_n}|} -  \frac{\alpha_k |I_k|}{|I_{y_n}|}
		\leq 
		\alpha_k \bigg| \frac{|I_k|}{|J_{I_k}(x_n)|} -  \frac{|I_k|}{|J_{I_k}(y_n)|} \bigg|
		= \alpha_k |T_{I_k}(x_n) - T_{I_k}(y_n)| = \alpha_k \Var_{[y_n, x_n]} (T_{I_k}).
		\]
		In the remaining case, $y_n < c(I_k) < x_n$, we add and subtract $\frac{\alpha_k |I_k|}{|I'|}$ with $I'$ such that $c(I') = c(I_k)$ and $r(I') = r(I_{y_n}) = r(I_{x_n})$. This gives 
		\[
		\alpha_k \bigg( \frac{|I_k|}{|I_{x_n}|} - \frac{|I_k|}{|I'|} + \frac{|I_k|}{|I'|} -  \frac{|I_k|}{|I_{y_n}|} \bigg)
		\leq 
		\alpha_k \bigg( \, \bigg| \frac{|I_k|}{|J_{I_k}(x_n)|} -  \frac{|I_k|}{|J_{I_k}(c(I_k))|} \bigg|
		+ \bigg| \frac{|I_k|}{|J_{I_k}(c(I_k))|} -  \frac{ |I_k|}{|J_{I_k}(y_n)|} \bigg| \, \bigg) 
		\]
		and the last quantity does not exceed 
		$\alpha_k \big(2T_{I_k}(c(I_k)) - T_{I_k}(x_n) - T_{I_k}(y_n)\big) = \alpha_k \Var_{[y_n, x_n]} (T_{I_k})$.
		
		Similarly, taking $y_1 \to -\infty$ and $I_{y_1}$ centered at $y$ and such that $r(I_{y}) = r(I_{x_1})$, we obtain
		\[
		Mf(x_1)
		= \lim_{y_1 \to -\infty} \bigg( \frac{1}{|I_{x_1}|} - \frac{1}{|I_{y_1}|} \bigg) \int_{I_{x_1}} f 
		\leq \lim_{y_1 \to -\infty} \sum_{k=1}^K \alpha_k \Var_{[y_1, x_1]} (T_{I_k})
		= \sum_{k=1}^K \alpha_k \Var_{(-\infty, x_1]} (T_{I_k})
		\]
		with the first summand omitted if $\alpha_1 \leq a$. As before, for $Mf(x_{N})$ the proof is analogous. 
	\end{proof} 
	    
	\begin{proof}[Proof of Theorem~\ref{T1}]
		For simplicity, we assume $a \geq b$. If $Mf$ has no local maxima, then 
		\[
		\Var(Mf) = \frac{a-b}{2} =\big(f(+\infty) - f(-\infty)\big) - \frac{a-b}{2} \leq \Var(f) - \frac{a-b}{2}.
		\]
		%The same conclusion easily follows when $Mf$ has exactly one representative. 
		Assume therefore that $N\ge1$ and $x_1 < \dots < x_N$ are all representatives. Since $a \geq b$, we obtain $Mf(x) \geq Mf(\infty)$ for all $x \in \RR$ and so $Mf$ decreases on $[x_N,\infty)$. Therefore, Lemma~\ref{L3} gives 
		\[
		\Var_{[x_N,\infty)}(Mf) = Mf(x_N) - \frac{a+b}{2}
		\leq \sum_{k=1}^K \alpha_k \Var_{[x_N, \infty)} (T_{I_k}) - \frac{a+b}{2}
		\]
		with the first summand omitted if $\alpha_1 \leq a$.
		If $Mf$ is nondecreasing on $(- \infty, x_1]$, then similarly
		\[
		\Var_{(-\infty, x_1]}(Mf) = Mf(x_1) - a 
		\leq \sum_{k=1}^K \alpha_k \Var_{(-\infty, x_1]} (T_{I_k}) - a.
		\]
		In this case, combining \eqref{var-split} with the two estimates above and Lemma~\ref{L3} gives \eqref{eq:T1}, since
		\[
		\Var(Mf) \leq \sum_{k=1}^K \alpha_k \Var(T_{I_k}) - a - \frac{a+b}{2} = \sum_{k=1}^K 2 \alpha_k - a - b - \frac{a-b}{2}
		\leq \Var(f) - \frac{a-b}{2}.
		\]
		If 
		%$x_1$ is the first representative but 
		$Mf$ is not nondecreasing on $(- \infty, x_1]$, then 
		$Mf$ has a local minimum at $y_1 \in [A, x_1)$ such that 
		\[
		\Var_{(-\infty, x_1]}(Mf) = a - 2 Mf(y_1) + Mf(x_1)
		\leq \min \Big \{ Mf(x_1) - b, a - \frac{a+b}{2} + \sum_{k=1}^K \alpha_k \Var_{[y_1, x_1]} (T_{I_k}) \Big \}. 
		\]
		The above inequality, which splits into two parts, follows from $(a+b)/2\le Mf$ (first part) and the estimate for $Mf(x_1)$ from Lemma~\ref{L3} with the additional nonnegative term  $2Mf(y_1)-(a+b)/2$ added (second part). Again the first summand is not used if $\alpha_1 \leq a$.
		%We now analyze separately three cases depending on the relation between $A$, $l(I_1)$, $a$, and $\alpha_1$.
		%If $l(I_1) > A$, then we choose the second term from the above minimum and obtain \eqref{eq:T1}, since after splitting $\Var(Mf)$ with respect to $(-\infty, x_1]$, $[x_n,x_{n+1}]$, $n =1, \dots, N-1$, and $[x_N, \infty)$, we obtain
		%\begin{align*}
		%\Var(Mf)
		%\leq a -  \frac{a+b}{2} + \sum_{k=1}^K \alpha_k \Var(T_{I_k}) -  \frac{a+b}{2}  = \sum_{k=1}^K 2 \alpha_k - b
		%\leq \Var(f) - a.
		%\end{align*}
		Now, if $\alpha_1 \leq a$, then choosing the second term from the above minimum and applying \eqref{var-split} and Lemma~\ref{L3}, we obtain
		\[
		\Var(Mf) \leq
		\sum_{k =2}^K 2 \alpha_k - b = (a - \alpha_1) + \Big( \alpha_1 + \sum_{k =2}^{K-1} 2 \alpha_k + \alpha_K \Big) + (\alpha_K - b) - a \leq \Var(f) - a 
		\]
		by the separation condition $\alpha_k \alpha_{k+1} = 0$. Finally, if $\alpha_1 > a$, then we choose the first term from the relevant minimum. Now $x_1 > r(I_1)$, since otherwise we would have 
		$Mf(y_1) > a$ by taking $I_{y_1}$ with $r(I_{y_1}) = r(I_1)$. Also, $Mf(x_1) < \alpha_1$ cannot occur because $Mf(r(I_1)) \geq \alpha_1 > Mf(-\infty)$, while $x_1$ is the leftmost representative. Similarly, $Mf(x_1) = \alpha_1$ would imply   
		$Mf(y) = \alpha_1$ for all $y \in [c(I_1),x_1]$ due to $x_1$ being the leftmost representative, leading to the contradiction $Mf(y_1) > a$. Thus, $l(I_{x_1}) = l(I_{k'})$ for some $k'$ with 
		$\alpha_{k'} > \alpha_1 > a$ and $c(I_{k'}) \leq x_1$ by Lemma~\ref{L2}. We then obtain
		\[
		\Var(Mf) \leq \alpha_{k'} - b + \sum_{k=1}^K \alpha_k \Var_{[x_1, \infty)} (T_{I_k}) - \frac{a+b}{2} \leq (a-\alpha_1) + \Var(f) - \frac{a+b}{2} < \Var(f) - \frac{a-b}{2}, 
		\]
		since by $x_1 \geq c(I_{k'}) > c(I_1)$ the sum above does not exceed $2(\alpha_1 + \dots + \alpha_K)-\alpha_1 - \alpha_{k'}$.   
	\end{proof}

	\section{Further comments} \label{S4}
Here we collect several thoughts that complement our main results. We verify optimality of the constants in Conjecture~\ref{conj2}, discuss variants of 
\eqref{eq:T1} and \eqref{eq:T1a} for other operators, comment on the special case of indicator functions, and point out limitations of our proof method.   

First, we note that the constants $1$ and $\frac12$ in \eqref{eq:conj2} are optimal. Indeed, for $a,b \in [0,\infty)$ with $a \geq b$ and $N \in \NN$, take $g_N \in \BV_a^b(\RR)$ defined by 
$
g_N \coloneqq a \ind{(-\infty,-1)} + N \ind{[-1,1)} + b \ind{[1,\infty)}.
$ 
Then for $N \geq a$, a simple calculation shows $\Var(g_N) = 2N-a-b$ and $\Var(M g_N) = 2N - a - \frac{a+b}{2}$. Letting $N \to \infty$, we obtain $C(\RR) \geq 1$ and $c_a^b(\RR) \leq \frac{a-b}{2}$.

Next, we show that certain variants of Conjecture~\ref{conj2} 
%and its discrete counterpart 
hold for other maximal operators. Set
\begin{align*}
\overrightarrow{M}g(x) \coloneqq \sup_{r>0} \frac{1}{r} \int_{x}^{x+r} |g| 
\quad \text{and} \quad
\overrightarrow{M}G(n) \coloneqq \sup_{m\ge0} \frac{1}{m+1} \sum_{y  = n}^{n + m} |G(y)|,
\end{align*}
and, similarly,
\begin{align*}
\widetilde{M}g(x) \coloneqq \sup_{r_1,r_2>0} \frac{1}{r_1+r_2} \int_{x-r_1}^{x+r_2} |g|
\quad \text{and} \quad
\widetilde{M}G(n) \coloneqq \sup_{m_1,m_2\ge0} \frac{1}{m_1+m_2+1} \sum_{y  = n-m_1}^{n + m_2} |G(y)|.
\end{align*}
We call $\overrightarrow{M}$ and $\widetilde{M}$ the one-sided and uncentered maximal operators, respectively. When $\KK = \ZZ$, one can instead consider a variant of $\widetilde{M}$ allowing only
$m_1,m_2 \geq 0$ with $m_1 + m_2 +1 \in 2\NN-1$ so that the corresponding intervals are balls in $\ZZ$.
This does not change the conclusion below.  

Since the sampling procedure works in the context of $\overrightarrow{M}$ or $\widetilde{M}$ with only minor adjustments, for all $a,b \in \CC$ the associated constants $\overrightarrow{C}(\RR),  \overrightarrow{c_a^b}(\RR), \widetilde{C}(\RR), \widetilde{c_a^b}(\RR)$ are not worse than their discrete counterparts. Furthermore, when $\KK = \ZZ$ and $a,b \in [0,\infty)$, it is natural to conjecture that $\overrightarrow{C}(\ZZ) = \widetilde{C}(\ZZ) = 1$, $ \overrightarrow{c_a^b}(\ZZ) = \max\{ 0, b-a\}$, and $\widetilde{c_a^b}(\ZZ) = |b-a|$. Indeed, a suitable variant of \eqref{eq:T1a} follows because for each nonnegative $G \in \BV_a^b(\ZZ)$ one has $\overrightarrow{M}G \geq \max\{G,b\}$ and $\widetilde{M}G \geq \max\{G,a,b\}$, while the equality $\overrightarrow{M}G(n) = G(n)$ or $\widetilde{M}G(n) = G(n)$ holds whenever $\overrightarrow{M}G$ or $\widetilde{M}G$ has a local maximum
at $n$. Considering $G_N \in \BV_a^b(\ZZ)$ given by $G_N \coloneqq a \ind{-\NN} + N \ind{\{0\}} + b \ind{\NN}$ and $g_N \in \BV_a^b(\RR)$ given by $g_N \coloneqq \sum_{n \in \ZZ} G_N(n) \ind{[n-1/2,n+1/2)}$, we verify that the postulated constants are optimal. 

\begin{remark} \label{R2}
	For all $a,b \in [0,\infty)$ we have
	\begin{enumerate}
		\item[{\rm(1)}] $
		\overrightarrow{C}(\RR) = \overrightarrow{C}(\ZZ) = 1$
		and $\overrightarrow{c_a^b}(\RR) = \overrightarrow{c_a^b}(\ZZ) = \max\{ 0, b-a\}$,
		\item[{\rm(2)}] $\widetilde{C}(\RR) = \widetilde{C}(\ZZ) = 1$
		and $\widetilde{c_a^b}(\RR) = \widetilde{c_a^b}(\ZZ) = |b-a|$.
	\end{enumerate}
\end{remark}

In the case of indicator functions our proof of \eqref{eq:T1} and \eqref{eq:T1a} is simpler. Indeed, let $f \in \BV(\RR)$ be as in Theorem~\ref{T1} with 
$a,b,\alpha_k \in \{0,1\}$. Now, if $Mf$ has a local maximum at $x$, then either $f(x)=0$ or $x$ can be adjusted to $x_* = c(I_k)$ for some $k$. 
Moreover, if $Mf$ has a local minimum at $y$, then $f(y)=0$. Thus, we can replace $T_{I_k}$ by $M \ind{I_k}$, since both functions coincide on 
$\{c(I_k)\} \cup (\RR \setminus I_k)$. In particular, if $Mf$ is monotone on the closed interval $J$ with endpoints $x, y$ as before, then
	\[
	\Var_J(Mf) = Mf(x) - Mf(y) \leq \frac{1}{|I_x|} \int_{I_x} |f| - \frac{1}{|I_y|} \int_{I_y} |f| \leq \sum_{k =1}^K \alpha_k \Var_{J} (M \ind{I_k}). 
	\]  
	Using this and relevant estimates at $\pm \infty$ yields \eqref{eq:T1}, and \eqref{eq:T1a} follows 
	by the extension procedure.

Finally, notice that the method presented in the proof of Theorem~\ref{T1} has some natural limitations. Indeed, during the proof we always choose $I_y$ to be the smallest interval centered at $y$ which contains $I_x$. This choice is sometimes very far from optimal. To illustrate this obstacle, let us take $g \in \BV_0^0(\RR)$ given by $g \coloneqq \ind{[-2,-1) \cup [1,2)} + C \ind{[-2,2)}$ with $C = 10^6$, say. Then $Mg$ has local maxima at $-\frac{3}{2}, 0, \frac{3}{2}$ with the related values being $C+1, C+\frac{1}{2}, C+1$, and local minima at $-\frac{1}{2}, \frac{1}{2}$ with the related values being $C+\frac{1}{3},C+\frac{1}{3}$. Thus, $\Var(Mg) = 2C + 4 - \frac{1}{3} \leq 2C + 4 = \Var(g)$, as desired. However, for $x = 0$ we have $I_x = (-2,2)$ so that for $y = \frac{1}{2}$ we should take $I_y = (-2,3)$ in Theorem~\ref{T1}. Unfortunately, the difference between the corresponding averages is much larger than $Mg(0) - Mg(\frac{1}{2})$, therefore $\Var(Mg) \leq \Var(g)$ cannot be verified this way.

\subsection*{Acknowledgments} Partial supports for this research were given to Paul Hagelstein by a Simons Foundation grant (\#521719 to Paul Hagelstein), to Dariusz Kosz by a National Science Centre of Poland grant (SONATA BIS 2022/46/E/ST1/00036), and to Krzysztof Stempak by ZUS, Poland.
We also thank the anonymous referee for providing valuable suggestions and observations, which significantly improved the presentation of the results.

\end{document}